\documentclass{article}[12pt]

\usepackage{amsmath,amssymb,amscd,amsthm}
\usepackage{latexsym}
\usepackage{amsmath,amsfonts}
\usepackage{dsfont}
\usepackage{enumerate}
\usepackage{cite}
\usepackage{tikz}
\usepackage{hyperref}

\theoremstyle{plain}
\newtheorem{theorem}{Theorem}[section]
\newtheorem{corollary}[theorem]{Corollary}
\newtheorem{proposition}[theorem]{Proposition}

\theoremstyle{definition}

\newtheorem{remark}[theorem]{Remark}
\newtheorem{example}[theorem]{Example}

\newcommand{\I}{I(X,R)}
\newcommand{\FI}{FI(X,R)}
\newcommand{\D}{D(X,R)}
\newcommand{\Der}{\mathcal{D}}
\newcommand{\ADer}{\mathcal{AD}}
\newcommand{\IDer}{\mathcal{ID}}

\DeclareMathOperator{\Hom}{Hom}
\DeclareMathOperator{\End}{End}

\begin{document}

\title{\textbf{Derivations of the idealization of the incidence module over the finitary incidence algebra}}
\author{\'{E}rica Z.~Fornaroli\thanks{Corresponding author: ezancanella@uem.br} \ and Roger E.~M.~Pezzott \\
\\
\textit{{\normalsize Departamento de Matem\'atica, Universidade Estadual de Maring\'a,}} \\ \textit{{\normalsize Avenida Colombo, 5790, Maring\'a, PR, 87020-900, Brazil.}}}
\date{}

\maketitle

\begin{abstract}
Let $R$ be a unital commutative ring and $X$ a partially ordered set. Let $\FI$ and $\I$ be the finitary incidence algebra and the incidence module of $X$ over $R$, respectively, and let $\D=\FI\oplus \I$ be the idealization of the $\FI$-bimodule $\I$. We give a description of the Lie algebra of derivations of $\D$ and also provide some necessary or sufficient conditions on $R$ and $X$ for the description of this Lie algebra to be simpler. \\



\end{abstract}


\section*{Introduction}

Let $X$ be a partially ordered set (poset, for short) and let $R$ be a unital commutative ring. The \emph{incidence module} $\I$ of $X$ over $R$ is the $R$-module of functions $f:X\times X\to R$ such that $f(x,y)=0$ if $x\nleq y$. Consider the submodule $\FI$ formed by all functions $f\in \I$ such that for any $x,y\in X$, $x<y$, there is only a finite number of subintervals $[u,v]\subseteq[x,y]$ such that $u\neq v$ and $f(u,v)\neq 0$. This is an $R$-algebra under the convolution product
		$$(fg)(x,y)=\sum_{x\leq z\leq y}f(x,z)g(z,y),$$
for any $f, g\in \FI$, called the \emph{finitary incidence algebra} of $X$ over $R$. Moreover, $\I$ is an $\FI$-bimodule~\cite{KN}. When every interval of $X$ is finite, $X$ is called locally finite and, in this case, $\FI=\I$ is known as the \emph{incidence algebra} of $X$ over $R$~\cite{Rota}. If $X=\{x_i : i\in I\}$ is a chain with $I=\{1,\ldots, n\}$ $(n\in\mathds N$, $n\geq 1$) or $I=\mathds N$, such that $x_i\leq x_j \Leftrightarrow i\leq j$, then $\FI\cong UT_{|I|}(R)$, the algebra of upper triangular $|I|\times |I|$ matrices over $R$. If $X$ is an anti-chain, then $\FI\cong\prod_{|X|} R$.

Since $\I$ is an $\FI$-bimodule, we can consider the idealization $\D:=\FI\oplus \I$ of $\I$ (also called trivial extension of $\FI$ by $\I$). This idealization has been studied since 2012 in the case where $R$ is a field. Zassenhaus properties of $\D$ were studied by Dugas~\cite{DMH}, automorphisms by Dugas and Wagner~\cite{DW}, and anti-isomorphisms and involutions by the authors~\cite{ER22}.

Derivations of $\FI$ and generalizations were studied in \cite{Backawski,KD,KKW,mK2018,mK2021,Koppinen,ER} and others. Derivations from $\I$ to $\FI$ were described in \cite{DW} in the case where $R$ is a field.

In Section~\ref{sec:pre}, we fix the notation and recall some results that will be used throughout the work. We also give a sufficient condition for all derivations of $\FI$ and all derivations from $\FI$ to $\I$ to be inner (Corollary~\ref{all-comparable-component-der=inner}).

In Section~\ref{sec:derivations}, we describe the $R$-module $\Der(\D)$ of the derivations of $\D$. More precisely, we give a decomposition of $\Der(\D)$ as a direct sum of four $R$-submodules (Theorem~\ref{TeoDecDer}). To obtain such a decomposition, we describe all $\FI$-bimodule endomorphisms of $\I$ (Theorem~\ref{bimodule_hom}). Furthermore, under certain hypotheses on $X$ or $R$, we show that the decomposition of $\Der(\D)$ obtained earlier can be simplified (Corollaries~\ref{simpler}, \ref{derwithallcomparable} and \ref{ecc}).


\section{Preliminaries}\label{sec:pre}

\subsection{Modules and idealization}

We begin by establishing some notation and recalling some results regarding modules and idealization of bimodules.

Let $R$ be a unital commutative ring. Throughout the paper $R$-algebras are associative with unity, and the center of an $R$-algebra $A$ is denoted by $Z(A)$.

Let $A$ be an $R$-algebra and let $M, N$ be left $A$-modules. So $M$ and $N$ are also $R$-modules, and we use the notation $\Hom_R(M,N)=\{\phi: M \to N : \phi$ is $R$-linear$\}$ and $\End_R(M)=\Hom_R(M,M)$. Each $\Phi \in \End_R(M\oplus N)$ has a presentation
$$ \Phi= \begin{bmatrix}
\phi_{11}& \phi_{12}\\
\phi_{21}&  \phi_{22}
\end{bmatrix},$$
where $\phi_{11} \in \End_R(M)$, $\phi_{12} \in \Hom_R(N,M)$, $\phi_{21} \in \Hom_R(M,N)$, and $\phi_{22} \in \End_R(N)$, such that for any $\begin{bmatrix}
m\\n
\end{bmatrix} \in M\oplus N$,
$$ \Phi\left(\begin{bmatrix}
m\\n
\end{bmatrix} \right) = \begin{bmatrix}
\phi_{11}(m) + \phi_{12}(n)\\ \phi_{21}(m) + \phi_{22}(n)
\end{bmatrix}.$$

Given an $R$-algebra $A$ and an $A$-bimodule $M$, the direct sum of $R$-bimodules $A\oplus M$ is an $R$-algebra with multiplication
$$\begin{bmatrix}r\\m \end{bmatrix}\begin{bmatrix}s\\n \end{bmatrix} = \begin{bmatrix}rs\\rn +ms \end{bmatrix},$$
for all $r,s \in A$ and $m,n \in M$, which is called \emph{idealization of the $A$-bimodule $M$} (or \emph{trivial extension of $A$ by $M$}). It is easy to see that $0_{A\oplus M} = \begin{bmatrix}0_A\\0_M \end{bmatrix}$ and $1_{A\oplus M} = \begin{bmatrix}1_A\\0_M \end{bmatrix}.$

For any $A$-bimodule $M$, we use the notation $C_{M}(A) = \{ m \in M: mr=rm, \; \forall r \in A\}$ and $C_{A}(M) = \{ r \in A: rm = mr, \; \forall m \in M\}$.

\subsection{Derivations from an algebra to a bimodule}

Recall that an $R$-linear map $D$ from an $R$-algebra $A$ to an $A$-bimodule $M$ is called a \emph{derivation} if $D(ab)=D(a)b+aD(b)$ for all $a,b\in A$. The set $\Der(A,M)$ of the derivations from $A$ to $M$ is an $R$-module under the operations
$$(D_1+D_2)(a)=D_1(a)+D_2(a) \ \text{ and } \ (rD_1)(a)=rD_1(a)$$
for all $D_1,D_2\in \Der(A,M)$, $r\in R$ and $a\in A$. For each $m\in M$, let $D_m: A\to M$ be defined by $D_m(a)=ma-am$ for all $a\in A$. Then $D_m\in \Der(A,M)$ and it is called the \emph{inner derivation defined by $m$}. It is easy to see that
$$D_m+D_n=D_{m+n} \ \text{ and } \ rD_m=D_{rm}$$
for $m,n\in M$ and $r\in R$, thus the set $\IDer(A,M)=\{ D_m : m\in M\}$ is a submodule of $\Der(A,M)$.

In the particular case where $M=A$ we use the simpler notation $\Der(A)$ and $\IDer(A)$ for $\Der(A,A)$ and $\IDer(A,A)$, respectively. It is well known that $\Der(A)$ is a Lie $R$-algebra with the commutator product $[D_1,D_2]=D_1\circ D_2-D_2\circ D_1$, and $\IDer(A)$ is an ideal of $\Der(A)$.

\subsection{$\Der(\FI,\I)$ and $\Der(\FI)$}

Recall that $\delta \in \FI$ given by $\delta(x,y)=0$, if $x\neq y$ and $\delta(x,x)=1$ for all $x\in X$, is the unity of $\FI$. For any $x\leq y$ in $X$, we denote by $e_{xy}$ the element of $\FI$ such that $e_{xy}(x,y)=1$ and $e_{xy}(u,v)=0$, if $(u,v)\neq (x,y)$. We will also write $e_x$ for $e_{xx}$. It is easy to verify that
\begin{equation}\label{eie}
    e_{x}ie_{y}=i(x,y)e_{xy}, \text{ for all } x\leq y \text{ and } i\in \I.
\end{equation}

An element $\tau\in \I$ such that $\tau(x,y)+\tau(y,z)=\tau(x,z)$ whenever $x \leq y \leq z$, determines a derivation $L_{\tau}:\FI\to \FI$ by $L_{\tau}(f)(x,y)=\tau(x,y)f(x,y)$, for all $f \in \FI$ and $x,y \in X$. Such derivations are called \emph{additive} and we will denote the $R$-submodule of $\Der(\FI)$ formed by all additive derivations by $\ADer(\FI)$. By \cite[Theorem~3.2]{DW},
\begin{equation} \label{eqDerDec}
     \Der(\FI,\I) = \IDer(\FI,\I)+\ADer(\FI).
\end{equation}

If $D\in \Der(\FI)$, then we also consider $D\in \Der(\FI,\I)$. Thus, there are $i,\tau\in\I$, with $\tau$ additive, such that $D=D_i+L_{\tau}$. Since $L_{\tau}\in \Der(\FI)$, then $D_i=D-L_{\tau}\in \Der(\FI)$. By \cite[Proposition~6]{DW}, $i\in \FI$ and therefore $D_i\in \IDer(\FI)$. So we have
\begin{equation} \label{eqDerDecFI}
     \Der(\FI) = \IDer(\FI)+\ADer(\FI)
\end{equation}
(which also follows from \cite[Theorem~5]{KD} in the case where $K=A$.)

The next proposition follows immediately from \eqref{eqDerDec} and \eqref{eqDerDecFI}.

\begin{proposition}\label{der-inn-iff}
One has $\Der(\FI,\I)=\IDer(\FI,\I)$ if and only if $\Der(\FI)=\IDer(\FI)$.
\end{proposition}

An element $x_0\in X$ is called \emph{all-comparable}, if for any $x\in X$, $x_0\leq x$ or $x\leq x_0$. Clearly, if $X$ has an all-comparable element, then $X$ is connected. The existence of an all-comparable element in $X$ is sufficient (but not necessary) for $\ADer(\FI)\subseteq \IDer(\FI)$, by \cite[Proposition~4.1]{ER22} (and the paragraph immediately following it). More generally, we have the following result.

\begin{proposition}\label{all-comparable-component}
Let $X=\bigsqcup_{j\in J}X_j$ be the decomposition of a poset $X$ into its connected components. Suppose that for each $j\in J$, $X_j$ is a poset with an all-comparable element. Then $\ADer(\FI)\subseteq\IDer(\FI)$.
\end{proposition}
\begin{proof}
For any $g\in \I$ we denote $g_j=g|_{X_j\times X_j}$ for every $j\in J$. Let $x_j$ be an all-comparable element of $X_j$, for each $j\in X_j$. Let $\tau\in \I$ be additive and define $f_j\in I(X_j,R)$ by
$$f_j(x,y) = \begin{cases}
-\tau(x,x_j) & \text{if } x = y\leq x_j\\
\tau(x_j,x) & \text{if } x_j \leq x= y\\
0 &\text{if } x\neq y
\end{cases}$$
for each $j\in J$. Obviously, $f_j \in FI(X_j,R)$. From the proof of \cite[Proposition~4.1]{ER22} it follows that $L_{\tau_j}=D_{f_j}$ for all $j\in J$. Let $f\in \FI$ be defined by $f(x,y)=f_j(x,y)$ for $x,y\in X_j$. If $g\in \FI$ and $x\leq y$ in $X$ are such that $x,y\in X_j$, then
\begin{align*}
L_{\tau}(g)(x,y) & = \tau(x,y)g(x,y)=\tau_j(x,y)g_j(x,y)=L_{\tau_j}(g_j)(x,y)=D_{f_j}(g_j)(x,y)\\
                 & = (f_jg_j-g_jf_j)(x,y)=(fg-gf)_j(x,y)=D_{f}(g)(x,y).
\end{align*}
Thus, $L_{\tau}=D_f$ and therefore $\ADer(\FI)\subseteq\IDer(\FI)$.
\end{proof}

\begin{corollary}\label{all-comparable-component-der=inner}
Let $X=\bigsqcup_{j\in J}X_j$ be the decomposition of a poset $X$ into its connected components. Suppose that for each $j\in J$, $X_j$ is a poset with an all-comparable element. Then $\Der(\FI,\I) = \IDer(\FI,\I)$ and $\Der(\FI) = \IDer(\FI)$.
\end{corollary}


\section{Derivations of $\D$}\label{sec:derivations}

In this section, we provide a decomposition of $\Der(\D)$ as a direct sum of four $R$-submodules. Furthermore, under certain hypotheses on $X$ or $R$, we show that such a decomposition can be simplified.

\begin{proposition}\label{TeoEquivalencia}
Let
\begin{equation}\label{matrix_de_D}
D=
\begin{bmatrix}
d_{11} & d_{12}\\
d_{21} & d_{22}
\end{bmatrix}
\in\End_R(\D).
\end{equation}
Then $D\in\Der(\D)$ if and only if the following conditions hold:
\begin{enumerate}
\item[(i)] $d_{11}\in\Der(\FI)$;
\item[(ii)] $d_{21}\in\Der(\FI,\I)$;
\item[(iii)] $d_{12}$ is an $\FI$-bimodule homomorphism that satisfies
\begin{equation}\label{eq.d12}
d_{12}(i)j+i\,d_{12}(j)=0, \ \forall\, i,j\in\I;
\end{equation}
\item[(iv)] for any $f\in\FI$ and $i\in\I$,
\begin{equation}\label{eq.d22}
d_{22}(fi)=d_{11}(f)i+fd_{22}(i) \ \text{ and } \ d_{22}(if)=d_{22}(i)f+id_{11}(f).
\end{equation}
\end{enumerate}
\end{proposition}
\begin{proof}
Assume that conditions \textit{(i)--(iv)} hold. For any
$\begin{bmatrix}
f \\ i
\end{bmatrix},
\begin{bmatrix}
g \\ j
\end{bmatrix}
\in\D$ we have
\begin{align*}
& D\left(\begin{bmatrix}
f \\ i
\end{bmatrix}\right)
\begin{bmatrix}
g \\ j
\end{bmatrix}
+
\begin{bmatrix}
f \\ i
\end{bmatrix}
D\left(\begin{bmatrix}
g \\ j
\end{bmatrix}\right)\\
 & \quad = \begin{bmatrix}
d_{11}(f)+d_{12}(i)\\
d_{21}(f)+d_{22}(i)
\end{bmatrix}
\begin{bmatrix}
g \\ j
\end{bmatrix}
+
\begin{bmatrix}
f \\ i
\end{bmatrix}
\begin{bmatrix}
d_{11}(g)+d_{12}(j)\\
d_{21}(g)+d_{22}(j)
\end{bmatrix}\\
& \quad =
\begin{bmatrix}
d_{11}(f)g+d_{12}(i)g+fd_{11}(g)+fd_{12}(j)\\
d_{11}(f)j+d_{12}(i)j+d_{21}(f)g+d_{22}(i)g+fd_{21}(g)+fd_{22}(j)+id_{11}(g)+id_{12}(j)
\end{bmatrix}\\
& \quad =
\begin{bmatrix}
d_{11}(f)g+fd_{11}(g)+fd_{12}(j)+d_{12}(i)g\\
d_{21}(f)g+fd_{21}(g)+d_{11}(f)j+fd_{22}(j)+d_{22}(i)g+id_{11}(g)+d_{12}(i)j+id_{12}(j)
\end{bmatrix}\\
& \quad =
\begin{bmatrix}
d_{11}(fg)+d_{12}(fj+ig)\\
d_{21}(fg)+d_{22}(fj+ig)
\end{bmatrix}\\
& \quad =
\begin{bmatrix}
d_{11} & d_{12}\\
d_{21} & d_{22}
\end{bmatrix}
\begin{bmatrix}
fg\\
fj+ig
\end{bmatrix}\\
& \quad = D\left(
\begin{bmatrix}
f \\ i
\end{bmatrix}
\begin{bmatrix}
g \\ j
\end{bmatrix}
\right),
\end{align*}
which shows that $D\in\Der(\D)$.

The converse follows directly from \cite[Theorem~2.1]{Ghosseiri}.
\end{proof}

\begin{remark}\label{fiel}
We observe that $\I$ is a faithful $\FI$-bimodule. Indeed, if $f\in \FI$ is such that $fi=0$ (resp.~$if=0$) for all $i\in \I$, then $f=f\delta=0$ (resp.~$f=\delta f=0$).
\end{remark}

\begin{corollary}\label{Coro.Equivalencia.Der}
Let
$$ D=
\begin{bmatrix}
d_{11} & d_{12}\\
d_{21} & d_{22}
\end{bmatrix}
\in\End_R(\D).$$
Then $D\in\Der(\D)$ if and only if conditions (ii), (iii), and (iv) of Proposition~\ref{TeoEquivalencia} hold.
\end{corollary}
\begin{proof}
Let $f,g\in\FI$ and $i\in \I$. If Proposition~\ref{TeoEquivalencia} \textit{(iv)} holds, then
$$d_{22}(fgi)=d_{11}(fg)i+fgd_{22}(i)$$
and also
$$d_{22}(fgi)=d_{11}(f)gi+fd_{22}(gi)=d_{11}(f)gi+fd_{11}(g)i+fgd_{22}(i).$$
Hence
$$d_{11}(fg)i=d_{11}(f)gi+fd_{11}(g)i$$
for all $i\in \I$. Thus, by Remark~\ref{fiel}, we have
$$d_{11}(fg)=d_{11}(f)g+fd_{11}(g)$$
for all $f,g\in\FI$, that is, $d_{11}\in\Der(\FI)$. Now, the result follows directly from Proposition~\ref{TeoEquivalencia}.
\end{proof}

\begin{remark}
Let $A$ be an $R$-algebra, $M$ an $A$-bimodule and consider the idealization $A\oplus M$. Clearly, Proposition~\ref{TeoEquivalencia} holds for $A$, $M$ and $A\oplus M$ instead of $\FI$, $\I$ and $\D$, respectively. Moreover, Corollary~\ref{Coro.Equivalencia.Der} is also true whenever $M$ is a faithful $A$-module.
\end{remark}

Let $D\in\Der(\D)$ be as in \eqref{matrix_de_D}. Then $d_{11}$ and $d_{21}$ are known by \eqref{eqDerDec} and \eqref{eqDerDecFI}. So, let us describe $d_{12}$ and $d_{22}$ more precisely. Note that, by Proposition~\ref{TeoEquivalencia}, $D$ admits a decomposition into the sum of three derivations of the form
\begin{equation}\label{eqDec1}
D=
\begin{bmatrix}
d_{11} & 0\\
0 & d_{22}
\end{bmatrix}+
\begin{bmatrix}
0 & 0\\
d_{21} & 0
\end{bmatrix}\\
+
\begin{bmatrix}
0 & d_{12}\\
0 & 0
\end{bmatrix}.
\end{equation}

The next result follows directly from Corollary~\ref{Coro.Equivalencia.Der}.

\begin{corollary}\label{propBimodule}
Let $\Phi=
\begin{bmatrix}
0 & 0\\
0 & \phi
\end{bmatrix}\in\End_R(\D)$. Then $\Phi\in\Der(\D)$ if and only if $\phi$ is an $\FI$-bimodule homomorphism.
\end{corollary}

\begin{proposition}\label{propDbarUnique}
Each $d\in\Der(\FI)$ extends to a unique $\overline d \in\End_R(\I)$ satisfying
\begin{equation}\label{eqDbarra}
\overline d(fi)=d(f)i+f\overline d(i) \ \text{ and } \ \overline d(if)=\overline d(i)f+id(f),
\end{equation}
for all $f\in\FI$ and $i\in\I$.
\end{proposition}
\begin{proof}
Let $d\in\Der(\FI)$. By \eqref{eqDerDecFI}, there exist $h\in\FI$ and an additive element $\tau\in\I$ such that $d=D_h+L_\tau$. We can naturally extend $D_h$ and $L_\tau$ to $\I$ by defining $\overline{D_h}(i)=hi-ih$ for all $i\in \I$, and $\overline{L_{\tau}}(i)(x,y)=\tau(x,y)i(x,y)$ for all $i\in \I$ and $x,y\in X$. Direct computations show that $\overline{D_h},\overline{L_{\tau}}\in \End_R(\I)$. Let $f\in\FI$ and $i\in\I$. We have
\begin{align}
D_h(f)i+f\overline{D_h}(i) & = (hf-fh)i+f(hi-ih)=h(fi)-(fi)h=\overline{D_h}(fi) \label{condicao_overlineDh1}\\
\overline{D_h}(i)f+iD_h(f) & = (hi-ih)f+i(hf-fh)=h(if)-(if)h=\overline{D_h}(if). \label{condicao_overlineDh2}
\end{align}
And for any $x,y\in X$,
\begin{align*}
(L_\tau(f)i+f\overline{L_\tau}(i))(x,y) & = (L_\tau(f)i)(x,y)+(f\overline{L_\tau}(i))(x,y)\\
& = \sum_{x\le z\le y} \tau(x,z)f(x,z)i(z,y) + \sum_{x\le z\le y} f(x,z)\tau(z,y)i(z,y)\\
& = \sum_{x\le z\le y} (\tau(x,z)+\tau(z,y))f(x,z)i(z,y)\\
& = \sum_{x\le z\le y} \tau(x,y)f(x,z)i(z,y) \\
& = \tau(x,y)(fi)(x,y)\\
& = \overline{L_\tau}(fi)(x,y).
\end{align*}
Therefore,
\begin{equation}\label{condicao_overlineLtau1}
\overline{L_\tau}(fi)=L_\tau(f)i+f\overline{L_\tau}(i).
\end{equation}
The identity
\begin{equation}\label{condicao_overlineLtau2}
\overline{L_\tau}(if)=\overline{L_\tau}(i)f+iL_\tau(f)
\end{equation}
is proved analogously.

Define $\overline d=\overline{D_h}+\overline{L_\tau}$. Clearly, $\overline d(f)=d(f)$ for all $f\in \FI$, and $\overline d\in \End_R(\I)$. Moreover, $\overline d$ and satisfies \eqref{eqDbarra} by \eqref{condicao_overlineDh1}, \eqref{condicao_overlineDh2}, \eqref{condicao_overlineLtau1}, and \eqref{condicao_overlineLtau2}.

To prove the uniqueness, let $\mathfrak d \in\End_R(\I)$ be an extension of $d$ that satisfies
$$\mathfrak d(fi)=d(f)i+f\mathfrak d(i) \ \text{ and } \ \mathfrak d(if)=\mathfrak d(i)f+id(f),$$
for all $f\in\FI$ and $i\in\I$. For any $i\in\I$ and $x\le y$, it follows by \eqref{eie} that
\begin{align*}
i(x,y)d(e_{xy}) & = i(x,y)\mathfrak d(e_{xy})=\mathfrak d(i(x,y)e_{xy})=\mathfrak d(e_xie_y)\\
                & = d(e_x)ie_y+e_x\mathfrak d(ie_y)=d(e_x)ie_y+e_x\mathfrak d(i)e_y+e_xid(e_y)\\
                & = d(e_x)ie_y+\mathfrak d(i)(x,y)e_{xy}+e_xid(e_y),
\end{align*}
thus
$$\mathfrak d(i)(x,y)=i(x,y)d(e_{xy})(x,y)-(d(e_x)i)(x,y)-(id(e_y))(x,y).$$
Since this last equality depends only on $d$, it follows that $\mathfrak d=\overline d$.
\end{proof}

\begin{proposition}\label{d_22-d_11}
If $\begin{bmatrix}
d_{11} & 0\\
0 & d_{22}
\end{bmatrix}\in \Der(\D)$, then $d_{22}-\overline{d_{11}}:\I\to\I$ is an $\FI$-bimodule homomorphism.
\end{proposition}
\begin{proof}
Since $d_{22},\overline{d_{11}}\in\End_R(\I)$, then $d_{22}-\overline{d_{11}}\in \End_R(\I)$. Let $f\in\FI$ and $i\in\I$. Applying \eqref{eq.d22} and \eqref{eqDbarra}, we obtain
\begin{align*}
(d_{22}-\overline{d_{11}})(fi) & = d_{11}(f)i+fd_{22}(i)-d_{11}(f)i-f\overline{d_{11}}(i)=f(d_{22}-\overline{d_{11}})(i),\\
(d_{22}-\overline{d_{11}})(if) & = d_{22}(i)f+id_{11}(f)-\overline{d_{11}}(i)f-id_{11}(f)=(d_{22}-\overline{d_{11}})(i)f.
\end{align*}
\end{proof}

Let $D=\begin{bmatrix}
d_{11} & d_{12}\\
d_{21} & d_{22}
\end{bmatrix}\in\Der(\D)$. By Corollary~\ref{propBimodule} and Propositions~\ref{propDbarUnique} and \ref{d_22-d_11}, $D$ admits a decomposition into the sum of four derivations since we can rewrite \eqref{eqDec1} as
\begin{equation}\label{eqDec2}
D=
\begin{bmatrix}
d_{11} & 0\\
0 & \overline{d_{11}}
\end{bmatrix}+
\begin{bmatrix}
0 & 0\\
0 & d_{22}-\overline{d_{11}}
\end{bmatrix}+
\begin{bmatrix}
0 & 0\\
d_{21} & 0
\end{bmatrix}\\
+
\begin{bmatrix}
0 & d_{12}\\
0 & 0
\end{bmatrix}.
\end{equation}
The maps $d_{12}$ and $d_{22}-\overline{d_{11}}$ are $\FI$-bimodule homomorphisms. Therefore, it is interesting to describe all the $\FI$-bimodule homomorphisms from $\I$ to $\I$. To do this, we recall that $Z(\FI)$ is the set of all diagonal functions which are constant on each connected component of $X$, and
\begin{equation} \label{eqcenralizadores}
    C_{\FI}(\I)= C_{\I}(\FI) = Z(\FI),
\end{equation}
by \cite[Proposition~2]{DW}. In particular, if $X$ is connected, then $\FI$ is a central algebra.

For each $z\in Z(\FI)$, we denote by $z\cdot:\I\to \I$ the multiplication by $z$. When $z=r\delta$ with $r\in R$, we just write $r\cdot$ instead of $(r\delta)\cdot$. Clearly, $z\cdot\in \End_R(\I)$ for all $z\in Z(\FI)$.

\begin{theorem}\label{bimodule_hom}
Let $\phi\in\End_R(\I)$. Then $\phi$ is an $\FI$-bimodule homomorphism if and only if $\phi=z\cdot$ for some $z\in Z(\FI)$. Furthermore, in this case, $z=\phi(\delta)$.
\end{theorem}
\begin{proof}
If $\phi=z\cdot$ for some $z\in Z(\FI)$, then $\phi(\delta)=z\delta=z$. Moreover, for any $f\in\FI$ and $i\in\I$,
\begin{align*}
\phi(fi) & = z(fi)=(zf)i=(fz)i=f(zi)=f\phi(i),\\
\phi(if) & = z(if)=(zi)f=\phi(i)f.
\end{align*}
Thus, $\phi$ is an $\FI$-bimodule homomorphism.

Conversely, if $\phi$ is an $\FI$-bimodule homomorphism, then $\phi(\delta)f=\phi(\delta f)=\phi(f)=\phi(f\delta)=f\phi(\delta)$
for all $f\in \FI$, and so $\phi(\delta)\in Z(\FI)$. Let $i\in\I$ and $x\le y$ in $X$. By \eqref{eie} we have
\begin{align*}
\phi(i)(x,y) & =(e_x\phi(i)e_y)(x,y)=\phi(e_xie_y)(x,y)=\phi(i(x,y)e_{xy})(x,y)=i(x,y)\phi(e_{xy})(x,y)\\
&=i(x,y)(e_{xy}\phi(\delta))(x,y)=i(x,y)\phi(\delta)(y,y)=(i\phi(\delta))(x,y)\\
&=(\phi(\delta)i)(x,y).
\end{align*}
Therefore,
$$\phi(i)=\phi(\delta)i=(\phi(\delta)\cdot)(i),$$
for every $i\in\I$, and hence $\phi=\phi(\delta)\cdot$.
\end{proof}

\begin{remark}\label{d_12andd_22-d_11}
Let $D\in\Der(\D)$ be as in \eqref{eqDec2}. By Theorem~\ref{bimodule_hom}, there exist $z,z_1\in Z(\FI)$ such that $d_{12}=z\cdot$ and $d_{22}-\overline{d_{11}}=z_1\cdot$. Moreover, since $d_{12}(\I)\subseteq \FI$, we have
\begin{equation}\label{condicao1z}
zi\in\FI, \ \forall i\in\I,
\end{equation}
and by Proposition~\ref{TeoEquivalencia},
\begin{equation}\label{condicao2zpre}
(zi)j+i(zj)=0, \ \forall i,j\in\I.
\end{equation}
However, \eqref{condicao1z} implies that $(zi)j=(iz)j=i(zj)$ for all $i,j\in\I$, therefore
\begin{equation}\label{condicao2z}
2i(zj)=0, \ \forall i,j\in\I.
\end{equation}
\end{remark}

Let $Z_0$ be the subset of $Z(\FI)$ formed by the elements $z$ that satisfy \eqref{condicao1z} and \eqref{condicao2z}. For each $x\in X$ we denote by $[x]$ the connected component of $X$ that contains $x$.

\begin{proposition}\label{propcaracterizaçãoZ0}
Let $z\in Z(\FI)$. Then $z\in Z_0$ if and only if
\begin{align}
z(x,x)=0, & \quad \text{if } [x] \text{ is not locally finite}; \label{eq[x]1}\\
2z(x,x)=0, & \quad \text{if } [x] \text{ is locally finite}. \label{eq[x]2}
\end{align}
\end{proposition}
\begin{proof}
Suppose that $z\in Z_0$. Consider the element $\Delta \in \I$ such that $\Delta(x,y)=1$ for all $x\leq y$ in $X$. If $[x]$ contains an infinite interval $[u,v]$ and $u<t<v$, then
$$(z\Delta)(u,t)=z(u,u)\Delta(u,t)=z(u,u).$$
On the other hand, since $u\in [x]$ and $z \in Z(\FI)$, then $z(u,u)=z(x,x)$. It follows that $(z\Delta)(u,t)=z(x,x)$ for $u<t<v$, which implies $z(x,x)=0$ as $z\Delta\in FI(X,R)$, by \eqref{condicao1z}. Therefore, \eqref{eq[x]1} holds. Now, if we apply \eqref{condicao2z} to $i=j=\delta$, we get $2z=0$, from which follows \eqref{eq[x]2}.

Conversely, suppose that $z$ satisfies \eqref{eq[x]1} and \eqref{eq[x]2}. Let $i\in I(X,R)$. If $[u,v]$ is an infinite interval of $X$ and $u\leq t< w \leq v$, then $[u,v]\subseteq [t]$. Hence, by \eqref{eq[x]1} we have
$$(zi)(t,w)=z(t,t)i(t,w)=0\cdot i(t,w)=0.$$
Therefore, $zi\in FI(X,R)$ and \eqref{condicao1z} holds. Since $z$ is diagonal, \eqref{eq[x]1} and \eqref{eq[x]2} imply that $2z=0$ and so $2i(zj)=0$ for all $i,j\in\I$. Thus, \eqref{condicao2z} is also true, and so $z \in Z_0$.
\end{proof}

Now, consider the following subsets of $\End_R(\D)$:
$$\widetilde{\mathcal D}=\left\{\begin{bmatrix}
d & 0\\
0 & \overline{d}
\end{bmatrix} : d\in\Der(\FI)\right\},$$
$$\mathcal Z_{22}=\left\{\begin{bmatrix}
0 & 0\\
0 & z\cdot
\end{bmatrix} : z\in Z(\FI)\right\},$$
$$\mathcal D_{21}=\left\{\begin{bmatrix}
0 & 0\\
d & 0
\end{bmatrix} : d\in\Der(\FI,\I)\right\},$$
$$\mathcal Z^{0}_{12}=\left\{\begin{bmatrix}
0 & z\cdot\\
0 & 0
\end{bmatrix} : z\in Z_0 \right\}.$$
By Corollary~\ref{Coro.Equivalencia.Der}, Corollary~\ref{propBimodule}, Proposition~\ref{propDbarUnique}, Theorem~\ref{bimodule_hom}, and Proposition~\ref{propcaracterizaçãoZ0}, $\widetilde{\mathcal D}$, $\mathcal Z_{22}$, $\mathcal D_{21}$, and $\mathcal Z^{0}_{12}$ are subsets of $\Der(\D)$.

\begin{proposition}\label{submodules_of_Der(D)}
    The subsets $\widetilde{\mathcal D}$, $\mathcal Z_{22}$, $\mathcal D_{21}$, and $\mathcal Z^{0}_{12}$ are $R$-submodules of $\Der(\D)$.
\end{proposition}
\begin{proof}
It is obvious that  $\mathcal D_{21}$ is a submodule of $\Der(\D)$.

Observe that $\overline{d_1+d_2}=\overline{d_1}+\overline{d_2}$ and $r\overline{d_1}=\overline{rd_1}$ for all $d_1,d_2\in\Der(\FI)$ and $r\in R$. Therefore, $\widetilde{\mathcal D}$ is an $R$-submodule of $\Der(\D)$.

Since $Z(\FI)$ is an $R$-subalgebra of $\FI$, then $\mathcal Z_{22}$ is an $R$-submodule of $\Der(\D)$.

By Proposition~\ref{propcaracterizaçãoZ0}, $Z_0$ is an $R$-subalgebra of $Z(\FI)$, then $\mathcal Z^{0}_{12}$ is an $R$-submodule of $\Der(\D)$.
\end{proof}

\begin{theorem}\label{TeoDecDer}
One has
$$\Der(\D)=\widetilde{\mathcal D}\oplus \mathcal D_{21}\oplus \mathcal Z^{0}_{12} \oplus \mathcal Z_{22}.$$
\end{theorem}
\begin{proof}
By \eqref{eqDec2}, Remark~\ref{d_12andd_22-d_11}, and Proposition~\ref{submodules_of_Der(D)} we have
\begin{equation}\label{sum}
    \Der(\D)=\widetilde{\mathcal D}+\mathcal D_{21}+\mathcal Z^{0}_{12}+\mathcal Z_{22}.
\end{equation}
Let $d_1\in \Der(\FI)$, $d_2\in \Der(\FI,\I)$, $z_1,z_2\in Z(\FI)$ with $z_1$ satisfying \eqref{condicao1z} and \eqref{condicao2z} so that
$$\begin{bmatrix}
d_1 & 0\\
0 & \overline{d_1}
\end{bmatrix}+
\begin{bmatrix}
0 & 0\\
d_2 & 0
\end{bmatrix}+
\begin{bmatrix}
0 & z_1\cdot\\
0 & 0
\end{bmatrix}+
\begin{bmatrix}
0 & 0\\
0 & z_2\cdot
\end{bmatrix}=
\begin{bmatrix}
0 & 0\\
0 & 0
\end{bmatrix}.$$
If follows directly that $d_1=0$ and $d_2=0$. By Remark~\ref{fiel}, $z_1=0$. By Proposition~\ref{propDbarUnique}, $\overline{d_1}=\overline{0}=0$, which implies $z_2=0$, by Remark~\ref{fiel}. Therefore, the sum in \eqref{sum} is direct.
\end{proof}

\begin{proposition}\label{propz12=0torsion}
One has $\mathcal Z^{0}_{12}=\{0\}$ if and only if $R$ is $2$-torsion-free or every connected component of $X$ contains an infinite interval.
\end{proposition}
\begin{proof}
We first observe that $\mathcal Z^{0}_{12}=\{0\}$ if and only if $Z_0=\{0\}$, by Remark~\ref{fiel}.

If $R$ is $2$-torsion-free or every connected component of $X$ contains an infinite interval, then $Z_0=\{0\}$ by Proposition~\ref{propcaracterizaçãoZ0}.

Conversely, suppose $R$ is not $2$-torsion-free and there exists a connected component $Y$ of $X$ that is locally finite. Let $r\in R$, $r\neq 0$, such that $2r=0$. Define $z\in \I$ by
$$z(x,y)=
\begin{cases}
r, & x=y\in Y,\\
0, & \text{otherwise}.
\end{cases}$$
Note that $z\in Z(FI(X,R))$ and $z\neq 0$. Let $x\in X$. If $x\notin Y$, then $z(x,x)=0$ while if $x\in Y$, then $2z(x,x)=2r=0$. Thus $z\in Z_0$, by Proposition~\ref{propcaracterizaçãoZ0}. Therefore, $Z_0\neq \{0\}$.
\end{proof}

The next result follows directly from Theorem~\ref{TeoDecDer} and Proposition~\ref{propz12=0torsion}.

\begin{corollary}\label{simpler}
One has
$$\Der(\D)=\widetilde{\mathcal D}\oplus \mathcal D_{21} \oplus \mathcal Z_{22}$$
if and only if $R$ is $2$-torsion-free or every connected component of $X$ contains an infinite interval.
\end{corollary}

We will now describe the inner derivations of $\D$. Let
$$\widetilde{\mathcal{ID}}=\left\{\begin{bmatrix}
D_h & 0\\
0 & \overline{D_h}
\end{bmatrix} : h\in \FI\right\} \ \text{ and } \
\mathcal{ID}_{21}=\left\{\begin{bmatrix}
0 & 0\\
D_l & 0
\end{bmatrix} : l\in\I\right\}.$$
Since $\IDer(\FI,\I)$ is a submodule of $\Der(\FI,\I)$ and $\IDer(\FI)$ is a submodule of $\Der(\FI)$, then $\widetilde{\mathcal ID}$ is a submodule of $\widetilde{\mathcal D}$ and $\mathcal{ID}_{21}$ is a submodule of $\mathcal{D}_{21}$.

\begin{proposition}\label{prop_inner}
One has
$$\IDer(\D)= \widetilde{\mathcal{ID}} \oplus \mathcal{ID}_{21}.$$
\end{proposition}
\begin{proof}
Let $\begin{bmatrix} h\\l \end{bmatrix}\in \D$. For all $\begin{bmatrix}f\\i \end{bmatrix}\in \D$ we have
\begin{align*}
    D_{\left[\begin{smallmatrix} h\\l \end{smallmatrix}\right]}\left(\begin{bmatrix}f\\i \end{bmatrix}\right) & = \begin{bmatrix}h\\l \end{bmatrix}\begin{bmatrix}f\\i \end{bmatrix}-\begin{bmatrix}f\\i \end{bmatrix}\begin{bmatrix}h\\l \end{bmatrix}=\begin{bmatrix} hf-fh\\hi+lf-fl-ih \end{bmatrix}\\
    & = \begin{bmatrix} D_h(f) \\ \overline{D_h}(i)+D_l(f)\end{bmatrix}=\begin{bmatrix} D_h & 0\\ D_l & \overline{D_h} \end{bmatrix}\begin{bmatrix}f\\i \end{bmatrix}.
\end{align*}
Therefore,
$$D_{\left[\begin{smallmatrix} h\\l \end{smallmatrix}\right]}=\begin{bmatrix} D_h & 0\\ D_l & \overline{D_h} \end{bmatrix}$$
and so $\IDer(\D)=\widetilde{\mathcal ID}\oplus \mathcal{ID}_{21}$.
\end{proof}

\begin{proposition}\label{allcomparable-inner}
Let $X=\bigsqcup_{j\in J}X_j$ be the decomposition of a poset $X$ into its connected components. Suppose that for each $j\in J$, $X_j$ is a poset with an all-comparable element. Then $\widetilde{\mathcal{D}}=\widetilde{\mathcal{ID}}$ and $\mathcal{D}_{21}=\mathcal{ID}_{21}$.
\end{proposition}
\begin{proof}
It follows from Corollary~\ref{all-comparable-component-der=inner}.
\end{proof}

The next result follows directly from Theorem~\ref{TeoDecDer} and Propositions~\ref{prop_inner} and \ref{allcomparable-inner}.

\begin{corollary}\label{derwithallcomparable}
Let $X=\bigsqcup_{j\in J}X_j$ be the decomposition of a poset $X$ into its connected components. Suppose that for each $j\in J$, $X_j$ is a poset with an all-comparable element. Then
$$\Der(\D)=\IDer(\D)\oplus \mathcal Z^{0}_{12} \oplus \mathcal Z_{22}=\widetilde{\mathcal{ID}} \oplus \mathcal{ID}_{21}\oplus \mathcal Z^{0}_{12} \oplus \mathcal Z_{22}.$$
\end{corollary}

From Proposition~\ref{propz12=0torsion} and Corollary~\ref{derwithallcomparable} we obtain our last result.

\begin{corollary}\label{ecc}
Let $X=\bigsqcup_{j\in J}X_j$ be the decomposition of a poset $X$ into its connected components. Suppose that for each $j\in J$, $X_j$ is a poset with an all-comparable element. If $R$ is $2$-torsion-free or $X_j$ contains an infinite interval for every $j\in J$, then
$$\Der(\D)=\IDer(\D)\oplus \mathcal Z_{22}.$$
\end{corollary}

\begin{example}
The set of natural numbers $\mathds N$ under divisibility is a poset such that $\mathds N=[1,0]$. Thus, $\Der(D(\mathds N,R))=\IDer(D(\mathds N,R))\oplus \mathcal Z_{22}$ for any unital commutative ring $R$, by Corollary~\ref{ecc}. Furthermore, since $(\mathds N,\mid)$ is connected, then $FI(\mathds N,R)$ is a central algebra, thus $\mathcal Z_{22}=\left\{\begin{bmatrix}
0 & 0\\
0 & r\cdot
\end{bmatrix} : r\in R\right\}.$
\end{example}

\begin{example}
Consider the poset $X$ so that its decomposition into connected components is $X=\bigsqcup_{j=1}^{\infty}X_j$ with $X_j=\mathds R$ for all $j\geq 1$, where the set of real numbers $\mathds R$ is considered with the usual order. By Corollary~\ref{ecc}, $\Der(\D))=\IDer(\D)\oplus \mathcal Z_{22}$ for any unital commutative ring $R$.
\end{example}

\begin{example}
Consider the poset $X$ so that its decomposition into connected components is $X=X_1\sqcup X_2$, with $X_1=\mathds N$ and $X_2=\mathds R$ where $\mathds N$ and $\mathds R$ are considered with the usual orders. By Corollary~\ref{derwithallcomparable}, $\Der(\D)=\IDer(\D)\oplus \mathcal Z^{0}_{12} \oplus \mathcal Z_{22}$ for all any unital commutative ring $R$. Furthermore, $\Der(\D))=\IDer(\D)\oplus \mathcal Z_{22}$ if and only if $R$ is 2-torsion-free, by Proposition~\ref{propz12=0torsion}.
\end{example}

\section*{Statements and Declarations}

The authors report there are no competing interests to declare.

\end{document}